\documentclass[11pt]{article}
\usepackage{amsmath,amsthm,amsfonts,amssymb,enumerate,calc,iwona}
\usepackage[colorlinks=true,citecolor=black,linkcolor=black,urlcolor=blue]{hyperref}
\usepackage{algorithm}
\usepackage{algorithmic}   
 \floatname{algorithm}{Algorithm}

\usepackage[lmargin=33mm,rmargin=33mm,bmargin=29mm,tmargin=29mm]{geometry}
\usepackage[numbers,sort&compress]{natbib}

\renewcommand{\baselinestretch}{1.1}
\renewcommand{\thefootnote}{\fnsymbol{footnote}}	
\allowdisplaybreaks

\newcommand{\textbold}[1]{{\bf\boldmath #1}}

\newcommand{\Zbl}[1]{Zbl:\,\href{http://www.zentralblatt-math.org/zmath/en/search/?q=an:#1}{#1}}
\newcommand{\arXiv}[1]{arXiv:\,\href{http://arxiv.org/abs/#1}{#1}}
\newcommand{\msn}[1]{MR:\,\href{http://www.ams.org/mathscinet-getitem?mr=MR#1}{#1}}
\newcommand{\MSN}[2]{MR:\,\href{http://www.ams.org/mathscinet-getitem?mr=MR#1}{#1}}
\newcommand{\doi}[1]{doi:\,\href{http://dx.doi.org/#1}{#1}}
\newcommand{\cs}[1]{CiteSeer:\,\href{http://citeseerx.ist.psu.edu/viewdoc/summary?doi=#1}{#1}}

\theoremstyle{plain}
\newtheorem{thm}{Theorem}
\newtheorem{lem}[thm]{Lemma}

\theoremstyle{definition}

\begin{document}

\title{\bf A Note on Hadwiger's Conjecture}

\author{David~R.~Wood\,\footnotemark[1]}

\date{\today}

\maketitle

\footnotetext[1]{School of Mathematical Sciences, Monash University, Melbourne, Australia
  (\texttt{david.wood@monash.edu}). Research supported by  the Australian Research Council.}

\renewcommand{\thefootnote}{\arabic{footnote}}

Hadwiger's Conjecture \citep{Hadwiger43} states that every $K_{t+1}$-minor-free graph is $t$-colourable. It is widely considered to be one of the most important conjectures in graph theory; see \citep{Toft-HadwigerSurvey96} for a survey. If every $K_{t+1}$-minor-free graph has minimum degree at most $\delta$, then every $K_{t+1}$-minor-free graph is $(\delta+1)$-colourable by a minimum-degree-greedy algorithm. The purpose of this note is to prove a slightly better upper bound.

\begin{lem}
\label{Main}
Fix $t\geq 2$. Assume that every $K_{t+1}$-minor-free graph has minimum degree at most $\delta$, and that every $K_{t}$-minor-free  graph with exactly $\delta$ vertices has an independent set of size $\alpha$. Then every $K_{t+1}$-minor-free graph is $(\delta-\alpha+2)$-colourable.
\end{lem}

\begin{proof}
We prove, by induction on $n$, that every $n$-vertex $K_{t+1}$-minor-free graph is $(\delta-\alpha+2)$-colourable. 
The base case with $n=1$ is trivial since $\delta-\alpha+2\geq 2$. 
Let $G$ be a $K_{t+1}$-minor-free graph. 
Let $v$ be a vertex of minimum degree $d$ in $G$. 
Thus $d\leq \delta$. 
If $d=0$ then, by induction, $G-v$ is $(\delta-\alpha+2)$-colourable, and by assigning to $v$ any colour already in use, we obtain a $(\delta-\alpha+2)$-colouring of $G$, as desired. 
Now assume that $d\geq 1$. 

Let $H$ be the subgraph of $G$ induced by $N(v)$. 
Thus $H$ has $d$ vertices and no $K_t$-minor. 
Let $H'$ be the graph obtained from $H$ by adding $\delta-d$ isolated vertices. 
Then $H'$ has exactly $\delta$ vertices and $H'$ also has no $K_t$-minor. 
By assumption, $H'$ has an independent set of size $\alpha$. 
Thus $H$ has an independent set $T$ of size $\alpha-\delta+d$. 

Let $G'$ be the graph obtained from $G$ by contracting each edge $vw$ where $w\in T$ into a new vertex $z$. 
Since $G'$ is a minor of $G$, $G'$ is $K_{t+1}$-minor-free. Since $d\geq 1$, $G'$ has less vertices than $G$. 
By induction, $G'$ is $(\delta-\alpha+2)$-colourable. 
Colour each vertex in $T$ by the colour assigned to $z$.
Colour each vertex in $G-T-v$ by the colour assigned to the same vertex in $G'$. 
Of the $d$ neighbours of $v$, at least $\alpha-\delta+d$ have the same colour. 
Thus at most $d-(\alpha-\delta+d)+1=\delta-\alpha+1$ colours are present on the neighbours of $v$. 
Hence, at least one of the $\delta-\alpha+2$ colours is not assigned to a neighbour of $v$, and this colour may be assigned to $v$. Thus $G$ is $(\delta-\alpha+2)$-colourable.
\end{proof}

The next lemma summarises some results about independent sets in a $K_{t+1}$-minor-free graphs.
Part (a) is the original result in this direction by \citet{DM82}.
Part (b), which is strong when $t$ is small, is by the author \citep{Wood-DMTCS07}. 
Part (c), which builds upon a similar result by \citet{Fox10}, is due to \citet{BK11}. See  \citep{MM-DM87,BLW11,PT10,KS05,KPT05} for related results. 

\begin{lem}
\label{BigIndSet}
Every $K_{t+1}$-minor-free graph on $n$ vertices has an independent set of size $\alpha$, where  \\
\hspace*{5mm} \textup{(a)} \quad  $(2\alpha -1)t \geq n$\enspace,\\
\hspace*{5mm} \textup{(b)} \quad  $(2\alpha -1)(2t- 5) \geq 2n -5$ for $t\geq 5$\enspace,\\
\hspace*{5mm} \textup{(c)} \quad  $(2-\gamma)\alpha t\geq n$\enspace, \\
where $\gamma= (80-\sqrt{5392})/126 \approx  0.0521\dots$.
\end{lem}

\begin{thm}
\label{Corollary}
Fix $t\geq 6$. Assume that every $K_{t+1}$-minor-free graph $G$ has minimum degree at most $\delta$.
Then $$\chi(G)\leq \delta- \frac{2\delta -5}{4t - 14}+\frac{3}{2}$$
and $$\chi(G)\leq\delta\left(1-\frac{1}{(2-\gamma)(t-1)}\right)+2\enspace,$$
where $\gamma= (80-\sqrt{5392})/126 \approx  0.0521\dots$.
\end{thm}

\begin{proof}
Lemma~\ref{BigIndSet}(b) implies that every $K_t$-minor-free graph with $\delta$ vertices has an independent set of size $\alpha$, where $(2\alpha -1)(2(t-1)- 5) \geq 2\delta -5$. Thus $\alpha\geq \frac{2\delta-5}{4t-14}+\frac{1}{2}$. Lemma~\ref{Main} implies that $\chi(G)\leq \delta-\alpha+2\leq \delta-\frac{2\delta-5}{4t-14}+\frac{3}{2}$. Similarly, Lemma~\ref{BigIndSet}(c) implies the second result. 
\end{proof}

Note that \citet{Kostochka82,Kostochka84} and \citet{Thomason01,Thomason84} independently proved that  $\delta\leq ct\sqrt{\log t}$ is the best possible upper bound on the minimum degree of $K_{t+1}$-minor-free graphs. Thus such graphs are $ct\sqrt{\log t}$-colourable. Unfortunately, Theorem~\ref{Corollary} makes no asymptotic improvement to this result. 

\medskip
We now apply these results for particular values of $t$. 

\textbold{$t=2$:} $K_3$-minor-free graphs are exactly the forests, and every forest has a vertex of degree 1. Thus $\chi(G)\leq 1-1+2=2$  by Lemma~\ref{Main}, which is tight. 

\textbold{$t=3$:}  Every $K_4$-minor-free graph $G$ has minimum degree at most 2, and every 2-vertex graph has an independent set of size 1. Thus $\chi(G)\leq 2-1+2=3$  by Lemma~\ref{Main}, which is tight. 

\textbold{$t=4$:}  Every $K_5$-minor-free graph with at least 3 vertices has at most $3n-6$ edges \citep{Mader68}. 
Thus every $K_5$-minor-free graph  has average degree less than 6 and minimum degree at most 5. Every 5-vertex $K_4$-minor-free graph has an independent set of size 2. Thus $\chi(G)\leq 5-2+2=5$  by Lemma~\ref{Main}. The 4-colour theorem and Wagner's characterisation \cite{Wagner37} implies that $\chi(G)\leq 4$. 

\textbold{$t=5$:}  Every $K_6$-minor-free graph with at least 4 vertices has at most $4n-10$ edges \citep{Mader68}. Thus Every $K_6$-minor-free graph has average degree less than 8, and minimum degree at most 7. Every 7-vertex $K_5$-free graph has an independent set of size $2$. Thus every $K_6$-minor-free graph is 7-colourable by Lemma~\ref{Main}, which is inferior to the result by \citet{RST-Comb93} who proved that such graphs are 5-colourable. Note that it is open whether  every $K_6$-minor-free graph has minimum degree at most 6 (see \citep{BJW11}). 

\textbold{$t=6$:}  Every $K_7$-minor-free graph with at least $5$ vertices has at most $5n-15$ edges \citep{Mader68}.
Thus every $K_7$-minor-free graph  has average degree less than 10, and minimum degree at most 9. Every 9-vertex $K_6$-free graph has an independent set of size $2$. Thus every $K_7$-minor-free graph is 9-colourable  by Lemma~\ref{Main}. 
\citet{BA13} proved that every $K_7$-minor-free graph is 8-colourable. 

We conjecture that every $K_7$-minor-free graph has minimum degree at most $7$ (which would be tight for $K_{1,2,2,2,2}$). This conjecture would imply that every $K_7$-minor-free graph is 7-colourable. Hadwiger's conjecture says that $K_7$-minor-free graphs are 6-colourable. 

\textbold{$t=7$:}  Every $K_8$-minor-free graph with at least 5 vertices has at most $6n-20$ edges \citep{Jorgensen94}, 
Thus every $K_8$-minor-free graph has average degree less than 12, and minimum degree at most 11. Every 11-vertex $K_7$-free graph has an independent set of size $2$. Thus every $K_8$-minor-free graph is 11-colourable by Lemma~\ref{Main}. \citet{BA13} proved that every $K_7$-minor-free graph is $10$-colourable.

\citet{Jorgensen94} characterised the $K_8$-minor-free graphs with $6n-20$ edges as those obtained from copies of $K_{2,2,2,2,2}$ by pasting on 5-cliques. Such graphs have minimum degree 8. We conjecture that every $K_8$-minor-free has minimum degree at most 8, which would imply that such graphs are 8-colourable. Hadwiger's conjecture says that $K_8$-minor-free graphs are 7-colourable. 

\textbold{$t=8$:}  Every $K_9$-minor-free graph $G$ has at most $7n-27$ edges \citep{ST06}, has average degree less than 14, and minimum degree at most 13. Every 13-vertex $K_8$-free graph has an independent set of size $2$. Thus 
every $K_9$-minor-free graph is $13$-colourable by Lemma~\ref{Main}. 

\citet{ST06} characterised the $K_9$-minor-free graphs with $7n-27$ edges as those obtained from copies of $K_{2,2,2,3,3}$ and $K_{1,2,2,2,2,2}$ by pasting on 6-cliques. Such graphs have minimum degree at most 10. We conjecture that every $K_9$-minor-free has minimum degree at most 10, which would imply that such graphs are 10-colourable by Lemma~\ref{Main}. Hadwiger's conjecture says that $K_9$-minor-free graphs are 8-colourable. 

\textbold{$t=9$}:  Every $K_{10}$-minor-free graph with at least 10 vertices has at most $11n- 66$ edges \citep{Song04}. Thus every $K_{10}$-minor-free graph has average degree less than $22$, and minimum degree at most 21. By Lemma~\ref{BigIndSet}(b), every $K_9$-minor-free 21-vertex graph has an independent set of 3 vertices. Thus  every $K_{10}$-minor-free graph is $20$-colourable by Lemma~\ref{Main} (or directly by Theorem~\ref{Corollary}). 

\textbold{$t=10$:}  Every $K_{11}$-minor-free graph with at least 11 vertices has at most $13n- 89$ edges \citep{Song04}. Thus every $K_{11}$-minor-free graph has average degree less than $26$, and minimum degree at most 25. By Lemma~\ref{BigIndSet}(b), every $K_{10}$-minor-free 25-vertex graph has an independent set of 3 vertices. Thus  every $K_{11}$-minor-free graph is $24$-colourable by Lemma~\ref{Main} (or directly by Theorem~\ref{Corollary}). 

We emphasis that in the cases $t=9$ and $t=10$, while it is likely that all of the stated bounds are far from optimal\footnote{See Chapter 6 of the Ph.D.\ thesis of \citet{Song04} for a discussion of the likely extremal graphs}, the utility of our approach is evident, since we may take $\alpha=3$ in these cases. 


\def\soft#1{\leavevmode\setbox0=\hbox{h}\dimen7=\ht0\advance \dimen7
  by-1ex\relax\if t#1\relax\rlap{\raise.6\dimen7
  \hbox{\kern.3ex\char'47}}#1\relax\else\if T#1\relax
  \rlap{\raise.5\dimen7\hbox{\kern1.3ex\char'47}}#1\relax \else\if
  d#1\relax\rlap{\raise.5\dimen7\hbox{\kern.9ex \char'47}}#1\relax\else\if
  D#1\relax\rlap{\raise.5\dimen7 \hbox{\kern1.4ex\char'47}}#1\relax\else\if
  l#1\relax \rlap{\raise.5\dimen7\hbox{\kern.4ex\char'47}}#1\relax \else\if
  L#1\relax\rlap{\raise.5\dimen7\hbox{\kern.7ex
  \char'47}}#1\relax\else\message{accent \string\soft \space #1 not
  defined!}#1\relax\fi\fi\fi\fi\fi\fi}

\end{document}